\documentclass{amsart}
\usepackage[utf8]{inputenc}
\usepackage{amssymb, bbm}
\usepackage{mathrsfs}
\usepackage{hyperref}
\usepackage{float}
\usepackage{xcolor}
\usepackage[shortlabels]{enumitem}
\usepackage[thinc]{esdiff}

\newtheorem{theorem}{Theorem}[section]
\newtheorem{corollary}[theorem]{Corollary}
\newtheorem{lemma}[theorem]{Lemma}
\newtheorem{proposition}[theorem]{Proposition}

\theoremstyle{definition}
\newtheorem{definition}[theorem]{Definition}
\theoremstyle{remark}
\newtheorem{remark}[theorem]{Remark}

\newcommand{\Ob}{\operatorname{Ob}}
\newcommand{\Hom}{\operatorname{Hom}}
\newcommand{\Mor}{\operatorname{Mor}}

\newcommand{\cat}[1]{\mathbf{#1}}
\newcommand{\norm}[1]{\|{#1}\|}

\newcommand{\subfield}{\mathbb{K}}

\title[A formula for the categorical magnitude]{A formula for the categorical magnitude in terms of the Moore-Penrose pseudoinverse}

\author{Stephanie Chen}
\address{California Institute of Technology, Pasadena, CA 91125, USA.}
\email{schen7@caltech.edu}

\author{Juan Pablo Vigneaux}
\address{1200 E California Blvd, Math Dept. 253-37, California Institute of Technology, Pasadena, CA 91125, USA.}
\email{vigneaux@caltech.edu}


\thanks{Declarations of interest: none}

\date{\today}

\begin{document}

\maketitle

\begin{abstract}
    The magnitude of finite categories is a  generalization of the Euler characteristic. It is defined using the coarse incidence algebra of rational-valued functions on the given finite category, and a distinguished element in this algebra: the Dirichlet zeta function. The incidence algebra may be identified with the algebra of $n \times n$ matrices over the rational numbers, where $n$ is the cardinality of the underlying object set. The Moore-Penrose pseudoinverse of a matrix is a generalization of the inverse; it  exists and is unique for any given matrix over the complex numbers. In this article, we derive a new method for calculating the magnitude of a finite category,  using the pseudoinverse of the matrix that corresponds to the zeta function. The magnitude equals the sum of the entries of this pseudoinverse.
\end{abstract}

\section{Introduction}
In number theory, given two  functions $f,g:\mathbb N \to \mathbb C$ such that
\begin{equation}\label{eq:mobius_direct}
    \forall n\geq 1, \quad f(n) = \sum_{d|n} g(d),
\end{equation}
 a procedure known as \emph{Möbius inversion} expresses $g$ as a function of $f$, 
\begin{equation}\label{eq:mobius_inverse}
    \forall n\geq 1, \quad g(n) =\sum_{d|n}  f(d)\mu(d,n),
\end{equation}
by means of certain integer coefficients $\mu(d,n)$. Gian-Carlo Rota \cite{Rota} extended this technique to partially ordered sets (posets), provided they are locally finite;\footnote{This means that every  interval $[a,b] :=\{ x\in P\,:\, a\leq x\leq b\}$ is a finite set.} the case above corresponds to the poset of positive integers ordered by divisibility. More explicitly, given a locally finite poset $(P,\leq)$, one considers the algebra $R(P)$ of rational-valued functions $\xi, \,\psi,...$ on $P\times P$, with product given by $(\xi \psi)(a,b)=\sum_{c\,: a\leq c\leq b} \xi(a,c)\psi(c,b)$,  and a distinguished element $\zeta \in R(P)$, called the Dirichlet zeta function, defined by $\zeta(x,y) = 1$ if $x\leq y$ and $\zeta(x,y) =  0$ otherwise. Equation \eqref{eq:mobius_direct} above thus corresponds to a particular case of $\phi = \gamma \zeta$, where $\phi (m,n) := 1_{\{m=1\}}f(n)$ and $\gamma(m,n):= 1_{\{m=1\}} g(n)$ for all $m,n\in \mathbb N$. The inverse $\mu=\zeta^{-1}$ in $R(A)$ is the \emph{M\"obius function}; in particular, \eqref{eq:mobius_inverse} corresponds to $\gamma = \phi  \mu$ (see Prop. 3 and Ex. 3 in \cite[Sec. 3]{Rota}). One might show that $\mu(x,y)=0$ whenever $x\not\leq y$ \cite[Sec. 3, Prop. 1]{Rota}, hence inversion respects the ``incidence'' relations in $P$.  Moreover, if the poset $P$ is finite and has an initial object $0$ and terminal object $1$, Rota defined its \emph{Euler characteristic} $E:=1+\mu(0,1)$ and proved that it equals the alternating sum $1-\sum_{i=2}^\infty (-1)^i C_i$ of the numbers $C_i$ of  chains with $i$ elements stretched between $0$ and $1$ \cite[Sec. 3, Prop.~6]{Rota}. 

More recently, Tom Leinster revisited Rota's work and extended the notion of Euler characteristic to finite ordinary categories  \cite{Leinster}. Given such a category $\cat A$ (with a finite number of objects and morphisms), the incidence algebra $R(\cat A)$ consists of rational-valued functions on $\Ob(\cat A)\times \Ob(\cat A)$, again equipped with convolution product (see \eqref{eq: incidence_mult} below), and the zeta function satisfies $\zeta(a,b)=|\Hom(a,b)|$, where $|\cdot|$ denotes cardinality. Since every poset can be seen as a category, with $x\leq y$ translating into a unique morphism from $x$ to $y$, the definitions are compatible with Rota's.\footnote{One should point out that Rota's inversion is \emph{local}, in the sense that $\mu(x,y)$ can be computed inverting the restriction of $\zeta$ to any interval $[a,b]$ containing $x$ and $y$, see \cite[Prop. 4]{Rota}.} There is an important difference, though: in this more general case, the element $\zeta$ is not necessarily invertible in $R(\cat A)$. To deal with this problem, Leinster introduced the notions of weighting and a coweighting (see Definition \ref{def:weightings} below) and proved that if both a weighting and coweighting exist, their ``total weight'' coincides and is, by definition,  the \emph{Euler characteristic} or \emph{magnitude} $\chi(\cat A)$ of the category. It coincides with the topological Euler characteristic of the nerve of $\cat A$, provided the latter is well defined (see Remark \ref{rmk:topology}). 
If $\cat P$ is a finite poset with $0$ and $1$, seen as a category, and $\cat P'$ is the full subcategory without $0$ and $1$, then by direct computation one might prove that $\chi(\cat P')=1+\mu(0,1)$, where $\mu$ is the M\"obius inverse of $\cat P$, hence one recovers Rota's characteristic, see \cite[Prop. 4.5]{Leinster}. 

Upon enumerating  the objects of a finite category $\cat A$ (set $n=|\Ob (\cat A)|$), there is an algebra homomorphism between $R(\cat A)$ and the matrices $\text{Mat}_{n \times  n}(\mathbb Q)$ equipped with the usual matricial product. Every matrix with complex coefficients has a unique \emph{Moore-Penrose pseudoinverse}, characterized by a system of equations (see Theorem \ref{thm: penrose}). In particular, the $\zeta$ function is represented by a matrix $Z$, whose pseudoinverse  $Z^+$ is the matricial representation of a \emph{pseudo-M\"obius function}. We shall prove that one might use $Z^+$ to check if a given finite category has magnitude and subsequently compute it. The magnitude  arises as the sum of all the entries of the  pseudoinverse matrix $Z^+$, see Corollary \ref{cor:computation_Euler_char} below. This approach also yields new proofs of the identities such as $\chi(\cat A\times\cat  B) = \chi(\cat A)\chi(\cat B)$ and $\chi(\cat A+\cat B)=\chi(\cat A)+\chi(\cat B)$, based on standard matricial operations.

Our main result, Theorem \ref{thm:formula_magnitude_matrix}, gives a formula for the computation of the magnitude of a \emph{matrix} in terms of its pseudoinverse and as such applies in a broader generality. For instance, Leinster  also introduced the magnitude of enriched categories \cite{leinster2013magnitude}; in this case, each hom-set is an object of a monoidal category $(\cat V,\otimes)$, and the zeta function is given by 
 $\zeta(a,b)=||\Hom (a,b)||$, where $||\cdot||$ is now any monoid homomorphism from $(\cat V,\otimes)$ to the multiplicative monoid of a semiring $\subfield$,  invariant under isomorphisms in $\cat V$. Again, this zeta function might be represented as a matrix $Z$ with coefficients in $\subfield$, and it is possible to introduce (co)weightings and the magnitude as in the case of ordinary categories. When $\subfield$ is a subfield of $\mathbb C$,  the pseudoinverse $Z^+$ gives explicit formulas for the weightings and coweightings, provided they exist, and hence a formula for the magnitude of enriched categories. Our proofs at the end of Section \ref{sec:magnitude_via_pseudoinversion}, concerning the algebraic properties of the magnitude,  are also valid in this enriched setting. In particular, the results discussed in this paper are applicable to finite metric spaces, for which the matrix $Z$, called the similarity matrix in \cite{leinster2013magnitude}, is real-valued. 
 
The results presented here were part of an undergraduate research project conducted by Stephanie  Chen under the mentorship of Juan Pablo Vigneaux in the summer and fall of 2022. When we were about to make this text public, Leinster pointed out to us that Mustafa Akkaya and \"Ozg\"un \"Unl\"u had independently arrived at the formula for the magnitude of ordinary finite categories in terms of the Moore-Penrose pseudoinverse of the zeta function in \cite{akkaya2023euler}.  The formula makes sense for any finite category, hence it can be taken as a generalization of Leinster's definition;  Akkaya and \"Ozg\"un  prove that this generalized magnitude is also invariant under equivalence of categories. 

The rest of this article is organized as follows. In Sections 2 and 3, we give the requisite background on the magnitude of finite categories and the Moore-Penrose pseudoinverse. We present our main results in Section 4. 

\section{Magnitude of a finite category}\label{sec:magnitude}

In what follows, $\cat{A}$ is assumed to be a finite category with $|\Ob(\cat{A})| = n$ unless otherwise stated. 

\begin{definition}
Let $\cat{A}$ be a finite category. Its (coarse) \emph{incidence algebra} is the set 
$
   R(\cat{A})$ of functions  $ f:  \text{Ob}(\cat{A}) \times \text{Ob}(\cat{A}) \to \mathbb{Q}$, which is a rational vector space under pointwise addition and scalar multiplication of functions, equipped with the  \emph{convolution product}:
\begin{equation}
\label{eq: incidence_mult}
   \forall \theta, \phi \in R(\cat{A}), \, \forall a,c \in\Ob(\cat{A}),\quad   (\theta \phi)(a,c) = \sum_{b \in\Ob(\cat{A})} \theta(a,b)\phi(b,c).
\end{equation}
 The identity element for this product is  Kronecker's delta function $\delta$; it is given by the formula
\begin{equation}
    \forall a,b\in \text{Ob}(\cat A), \quad \delta(a,b) = \begin{cases}
        1& \text{ if } a=b\\
        0& \text{otherwise}
    \end{cases}.
\end{equation}
\end{definition}
\begin{remark}\label{rem:incidence_algs}
This terminology comes from \cite{notionsMobius}; it might be confusing because the arrows in the category (i.e. the ``incidence relations'') play no role.  There is a similar algebra of functions $f:\Mor \cat{A} \to \mathbb Q$, called the \emph{fine incidence algebra}. In the case of posets, which are the subject of  Rota's theory \cite{Rota}, the fine incidence algebra is a subalgebra of the coarser one; also, the   zeta function introduced below belongs to the fine algebra and its inverse, the M\"obius function, too  \cite[Thm. 4.1]{Leinster}. For the purposes of this note, only the coarse incidence algebra is needed and we refer to it simply as \emph{incidence algebra}.
\end{remark}

Note that as $\cat{A}$ is a finite category, its underlying object class is a finite set and thus we may fix an ordering on its objects. Suppose that $|\Ob(\cat{A})| = n$, and $\Ob(\cat{A}) = \{x_1, \ldots, x_n\}$. For any $f \in R(\cat{A})$, we may define a corresponding matrix $F \in \text{Mat}_{n \times  n}(\mathbb{Q})$ by $F_{i,j} = f(x_i,x_j)$. In particular, the $\delta$ function corresponds to the $n\times n$ identity matrix. Conversely, any matrix $G \in \text{Mat}_{n \times  n}(\mathbb{Q})$ determines an element $g \in R(\cat{A})$ by $g(x_i,x_j) = G_{i,j}$.
Moreover, the multiplication operation given by equation (\ref{eq: incidence_mult}) agrees with matrix multiplication: given $f,g \in R(\cat{A})$ with $F,G$ their corresponding matrices, we then have that $FG$ is the corresponding matrix of $fg$. Thus, after fixing an ordering of the objects of $\cat{A}$, we may identify $R(\cat{A})$ with the algebra of $(n\times n)$-matrices over $\mathbb{Q}$. We then have that an element $f \in R(\cat{A})$ is invertible (with two-sided inverse) if and only if its corresponding matrix $F$ is an invertible matrix.

\begin{definition}
The \emph{zeta function} $\zeta$ is an element of $ R(\cat{A})$ given  by $\zeta(a,b) = |\text{Hom}(a,b)|$ for all $a,b \in\Ob(\cat{A})$. If $\zeta$ is invertible (with two-sided inverse) in $R(\cat{A})$, we say that $\cat{A}$ \emph{has M\"{o}bius inversion} and denote its inverse, the \emph{M\"obius function}, by $\mu = \zeta^{-1}$. 
\end{definition}
\begin{remark}
Let $Z$ be the corresponding matrix of $\zeta$. Then, $\cat{A}$ has M\"{o}bius inversion if and only if $Z$ is an invertible matrix, in which case  $Z^{-1}$ corresponds to $\mu$.
\end{remark}

\begin{definition}(Incidence weightings)\label{def:weightings}
A \emph{weighting} on $\cat{A}$ is a function $k^{\bullet}: \Ob(\cat{A}) \to \mathbb{Q}$ such that 
\begin{equation}
\label{eq: weighting}
\sum_{b \in\Ob(\cat{A})} \zeta(a,b)k^b = 1
\end{equation}
for all $a \in\Ob(\cat{A})$. Similarly, a \emph{coweighting} on $\cat{A}$ is a function $k_{\bullet}: \Ob(\cat{A}) \to \mathbb{Q}$ such that 
\begin{equation}
\label{eq: coweighting}
\sum_{b \in\Ob(\cat{A})} \zeta(b,a)k_b = 1
\end{equation}
for all $a \in\Ob(\cat{A})$. Equivalently, one may view a coweighting on $\cat{A}$ as a weighting on $\cat{A}^{\text{op}}$.
\end{definition}
\begin{remark}
A finite category $\cat{A}$ need not admit a weighting or a coweighting. It is also possible for $\cat{A}$ to have more than one weighting or coweighting. Examples of both are given in \cite{Leinster}.
\end{remark}
 
In light of our earlier identification of $R(\cat{A})$ with the algebra of $n\times n$ matrices over $\mathbb{Q}$, we also reformulate the above definitions in terms of matrices. In what follows,  $\subfield$ denotes a subfield of $\mathbb{C}$.

\begin{definition}(Matricial weightings)
\label{def: weight_matrix}
Let $M$ be a matrix in $\text{Mat}_{m\times n}(\subfield)$. Denote by $1_{m} \in \text{Mat}_{m\times 1}(\subfield)\cong \subfield^m$ the column vector of ones. 
A \emph{matricial weighting} of $M$ is a column vector $w \in \subfield^n$ such that 
\begin{equation}
Mw = 1_{m}.
\end{equation}
 A \emph{matricial coweighting} of $M$ is a matricial weighting of the transpose $M^T$; equivalently, a column vector $z \in \subfield^m$ such that 
\begin{equation}
z^T M = 1_{n}^T,
\end{equation}
where ${z}^T$ denotes the  transposed vector.
\end{definition}

\begin{lemma}[See \protect{\cite[Lem~2.1]{Leinster}}]\label{lem:coincidence_sums}
Let $M$ be an element of  $\text{Mat}_{m\times n}(\subfield)$.  Suppose that both a weighting $w=(w_1,...,w_n)$ and a coweighting $z=(z_1,...,z_m)$ of $M$ exist. Then
\begin{equation*}
\sum_{i=1}^n w_i = \sum_{i=1}^m  z_i. 
\end{equation*}
\end{lemma}
\begin{proof}Simply remark that
    \begin{equation}
        \sum_{i=1}^n w_i = 1_n^T w =  z^T M w = z^T 1_m =   \sum_{i=1}^m  z_i.
    \end{equation}
\end{proof}

In view of this result, the following definition makes sense.

\begin{definition}[See \protect{\cite[Def. 1.1.3]{leinster2013magnitude}}]
 We say that $M\in \text{Mat}_{m\times n}(\subfield)$ \emph{has magnitude} if it admits both a weighting $w$ and a coweighting $z$, in which case we define the \emph{magnitude} $\chi(M)$ by
\begin{equation}
\chi(M) = \sum_{i=1}^n w_i = \sum_{i=1}^m z_i. 
\end{equation}
\end{definition}

Under the identification of $R(\cat A)$ and $\text{Mat}_{n\times n}(\mathbb{Q})$ described above, there is an obvious correspondence between incidence weightings of $\cat A$ and the matricial weightings of $Z$. We refer to the matricial weightings of $Z$ (resp. matricial coweightings of $Z$) as weighting vectors (resp. coweighting vectors) of $\cat{A}$. Invoking again Lemma \ref{lem:coincidence_sums}, we introduce the following definition.

\begin{definition}\label{def:magnitude}
Let $\cat{A}$ be a finite category. We say that $\cat{A}$ \emph{has magnitude} if it admits both a weighting $k^\bullet$ and a coweighting $k_\bullet$, in which case we define the \emph{magnitude} $\chi(\cat{A})$ by
\begin{equation}
\chi(\cat{A}) = \sum_{a\in\Ob(\cat{A})} k^a = \sum_{a\in\Ob(\cat{A})} k_a. 
\end{equation}
\end{definition}

 \begin{remark}\label{rmk:topology}
     The notation is justified by the following connection with topology: if $\cat A$ is a finite skeletal category containing no endomorphisms except identities, then the magnitude of $\cat A$ coincides with the (topologically defined) Euler characteristic of its classifying space $B\cat A$, which is the geometric realization of the nerve of $\cat A$ [2, Prop. 2.11]. In particular, if $\cat A$ is the incidence poset of a simplicial complex $S$, its classifying space is naturally identified with its barycentric subdivision and $\chi(\cat A)$ equals the topological Euler characteristic of $S$. For this reason, $\chi(\cat A)$ was initially called the \emph{Euler characteristic} of the category $\cat A$, see \cite{Leinster}.
 \end{remark}

It is clear that $\cat{A}$ admits a unique weighting vector if $Z$ is an invertible matrix. Conversely, if $\cat{A}$ admits a unique weighting vector, then the nullspace of $Z$ must have dimension $0$, hence $Z$ as a square matrix must be of full rank and thus invertible. 
Similarly, $\cat{A}$ admits a unique coweighting vector if and only if $Z$ is an invertible matrix. In this case, the weighting and coweighting vectors are given by $w = Z^{-1}1_n$ and $z = Z^{-T} 1_n$ respectively. Equivalently, we have that $\cat{A}$ admits a unique weighting (resp. coweighting) if and only if $\cat{A}$ has M\"{o}bius inversion, in which case the weighting and coweighting are given by
\begin{equation}
k^a = \sum_{b\in\Ob(\cat{A})} \mu(a,b), \text{ \ \ \ } k_b = \sum_{a\in\Ob(\cat{A})} \mu(a,b)
\end{equation}
respectively, and
\begin{equation}
\label{eq: mobius_euler}
\chi(\cat{A}) = \sum_{a,b\in\Ob(\cat{A})} \mu(a,b).
\end{equation}

\section{Moore-Penrose inverse}\label{sec:pseudoinverse}

Equation \eqref{eq: mobius_euler} gives a way of calculating the magnitude when $\cat{A}$ has M\"{o}bius inversion, i.e. the zeta matrix $Z$ is invertible. We shall see below that, when $Z$ is not invertible, there is a similar formula in terms of any generalized inverse of $Z$. Moreover, if one works with the Moore-Penrose pseudoinverse of $Z$, many properties of the magnitude can be deduced from matricial identities.  
\begin{theorem}[Penrose, \protect{\cite[Thm~1]{Penrose}}]
\label{thm: penrose}
Let $A$ be an  $(m\times n)$-matrix over $\mathbb{C}$. Then, there exists a unique $(n\times m)$-matrix $A^+$ over $\mathbb{C}$ satisfying the following four properties:
\begin{itemize}
\item[(i)] $AA^+A=A$
\item[(ii)] $A^+AA^+ = A^+$
\item[(iii)] $(AA^+)^* = AA^+$
\item[(iv)] $
(A^+A)^* = A^+A$.
\end{itemize} 
\end{theorem}
Concerning the notation, ${M}^*$ denotes the conjugate transpose of a matrix $M$.
\begin{definition} Let $A$ be an  $(m\times n)$-matrix over $\mathbb{C}$.
The matrix $A^+$ above is called the \emph{Moore-Penrose pseudoinverse} of $A$. More generally, an $(n\times m)$-matrix $X$ satisfying $AXA=A$ is called a \emph{generalized inverse} of $A$. 
\end{definition}

Unlike the Moore-Penrose pseudoinverse, a generalized inverse is unique only if $A$ is invertible \cite[Ch. 1, Sec. 2]{BenIsrael2003}. 

\begin{remark}\label{rmk:integral_domains}
    K. P. S. Bhaskara Rao and his collaborators \cite{Rao1983,Bapat1990} studied the existence of generalized inverses and the (necessarily unique) Moore-Penrose pseudoinverse for matrices over integral domains. They proved that  an $(m\times n)$-matrix $A$ of determinantal rank $r$ over an integral domain $R$ has a generalized inverse if and only if there are coefficients $c_{\alpha}^\beta \in R$, indexed by $r$-element subsets $\alpha$ of $\{1,...,m\}$ and $\beta$ of   $\{1,...,n\}$, such that 
    \begin{equation}
        \sum_{\alpha,\beta} c_\alpha^\beta \det A_\alpha^\beta  = 1,
    \end{equation}
    where $A_\alpha^\beta$ is the submatrix of $A$ determined by the rows indexed by $\alpha$ and columns indexed by $\beta$. 
    They  gave an explicit formula for the generalized inverse involving the minors $\det A_\alpha^\beta$  of $A$ and the coefficients $c_\alpha^\beta$ (which are in general nonunique). Finally, they proved that the Moore-Penrose pseudoinverse of $A$ exists if and only if $\sum_{\alpha,\beta}(\det A_\alpha^\beta)^2$ is an invertible element of $R$, and in this case also gave an explicit formula for the pseudoinverse. 
\end{remark}

Henceforth, let $\subfield$ denote a subfield of $\mathbb C$.
It follows from the uniqueness of the pseudoinverse in Theorem \ref{thm: penrose} that when a matrix $A$ over $\subfield$ is invertible, $A^+ = A^{-1}$. The pseudoinverse is straightforward to calculate if $A$ has either linearly independent columns or linearly independent rows. If $A$ has linearly independent columns, then $A^*A$ is invertible. One may then check that $(A^*A)^{-1}A^*$ satisfies all four properties in Theorem \ref{thm: penrose}, and thus $A^+ = (A^*A)^{-1}A^*$. Similarly, if $A$ has linearly independent rows, then $A^*$ has linearly independent columns, and thus $AA^*$ is invertible. Then, $A^+ = A^*(AA^*)^{-1}$. 

More generally,  consider $A \in \text{Mat}_{m\times n} (\subfield)$ and let $r\leq \min(m,n)$ be the rank of $A$, which means that the dimension of the column space of $A$ over $\subfield$ is $r$. We may then pick a basis $\{\mathbf{b_1}, \ldots, \mathbf{b_r}\}$ of the column space of $A$ and construct a matrix $B \in \text{Mat}_{m\times r}(\subfield)$ by $B = [\mathbf{b_1}, \ldots, \mathbf{b_r}]$. Note that then $B$ is of rank $r$. Every column of $A$ is then a linear combination of the columns of $B$. Denote the columns by $A = [\mathbf{a_1}, \ldots, \mathbf{a_n}]$. Then, for all $1\leq j \leq n$, $\mathbf{a_j} = \sum_{i =1}^r c_{i,j}\mathbf{b_i}$, with $c_{i,j} \in \subfield$ for all $1 \leq i \leq r$. Define the matrix $C\in \text{Mat}_{r \times n}(\subfield)$ by $C_{i,j} = c_{i,j}$. We then have that $A = BC$, so $r = \text{rank}(A) = \text{rank}(BC) \leq \text{rank}(C) \leq r$, and thus $\text{rank}(C) = r$. We thus have a \emph{rank decomposition} $A = BC$. The matrix $B$ has linearly independent columns and $C$ has linearly independent rows, so $B^+ = (B^*B)^{-1}B^*$ and $C^+ = C^*(CC^*)^{-1}$. Moreover  \cite[Thm~1]{rank_factor}, 
\begin{equation}\label{eq:formula_pseudoinverse}
    A^+ =  C^+B^+ = C^*(CC^*)^{-1}(B^*B)^{-1}B^* \in \text{Mat}_{n\times m}(\subfield).
\end{equation}
Hence the rank decomposition of a matrix gives a general method for calculating the pseudoinverse.

The explicit formula \eqref{eq:formula_pseudoinverse} entails the following property of the pseudoinverse.
\begin{lemma}
\label{lem: field_pseudo}
Let $\subfield$ be a subfield of $\mathbb{C}$. If $A$ is a matrix over $\subfield$, then $A^+$ is also a matrix over $\subfield$.
\end{lemma}

An alternate formula for calculating the Moore-Penrose pseudoinverse of a matrix uses its singular value decomposition (over $\mathbb C$). Recall that any matrix $M\in \text{Mat}_{m\times n}(\mathbb{C})$ admits a singular value decomposition given by $M = U\Sigma V$ where $U \in \text{Mat}_{m\times m}(\mathbb{C})$ and $V\in \text{Mat}_{n\times n}(\mathbb{C})$ are unitary matrices and $\Sigma \in \text{Mat}_{m\times n}(\mathbb{C})$ is a diagonal matrix with entries that are positive (hence real) or zero. Denote by $\Sigma' \in \text{Mat}_{n\times m}(\mathbb{C})$ the matrix obtained from $\Sigma$ by inverting the non-zero elements of $\Sigma$ and then taking the transpose. Then $\Sigma\Sigma', \Sigma'\Sigma$ are both diagonal square matrices whose nonzero entries are all 1. The reader may verify that $\Sigma'$ then satisfies the four properties of Theorem \ref{thm: penrose}|hence $\Sigma^+ = \Sigma'$| and that
\begin{equation}
\label{eq:SVD_pseudoinverse}
    M^+ = V^*\Sigma^+U^*.
\end{equation}
This appears as Lemma 1.6 in \cite{Penrose}.

We now consider the behavior of the pseudoinverse with respect to the Kronecker product and direct sum operations; this will entail properties of the magnitude under products and coproducts of categories in the next section. Take two matrices $A \in \text{Mat}_{m\times n}(\subfield), B\in \text{Mat}_{p\times q}(\subfield)$. Their Kronecker product $A \otimes B \in \text{Mat}_{pm \times qn}(\subfield)$ is given by the block matrix
\begin{equation}
A \otimes B
= \begin{pmatrix}
a_{1,1} B & \ldots&  a_{1,n} B\\
\vdots & \ddots & \vdots \\
a_{m,1}B & \ldots & a_{m,n}B 
\end{pmatrix}.
\end{equation}
Likewise, their direct sum $A \oplus B \in \text{Mat}_{(m+p)\times (n+q)}(\subfield)$ is given by the block matrix
\begin{equation}
A \oplus B = \begin{pmatrix}
A & 0\\
0 & B
\end{pmatrix}.
\end{equation}

For use in proving the following proposition, we state some well-known properties of the Kronecker product and direct sum of two matrices, all of which are straightforward to verify.
\begin{lemma}
\label{lem:properties_kroneckerprod_directsum}
 Denote by $I_n$ the $n\times n$ identity matrix. The identities $I_n\otimes I_m = I_{nm}$ and $I_n\oplus I_m = I_{n+m}$ hold. Moreover, for any matrices $A \in \text{Mat}_{m\times n}(\subfield)$, $ B\in \text{Mat}_{p\times q}(\subfield)$, $C\in \text{Mat}_{n\times l}(\subfield)$, and $ D\in \text{Mat}_{q\times r}(\subfield)$, 
\begin{enumerate}[(i)]
    \item $(A\otimes B)^* = A^* \otimes B^*$ and $(A\oplus B)^* = A^* \oplus B^*$; 
   
        \item  $(A\otimes B)(C\otimes D) = AC\otimes BD$; and 
        \item  $(A\oplus B)(C\oplus D) = AC\oplus BD$.
\end{enumerate}    
\end{lemma}

\begin{proposition}
\label{prop: kronecker_direct_sum}
 For any $A \in \text{Mat}_{m\times n}(\mathbb{\subfield})$ and $ B\in \text{Mat}_{p\times q}(\mathbb{\subfield})$,
\begin{itemize}
\item[(i)] $(A \otimes B)^+ = A^+ \otimes B^+$, 
\item[(ii)] $(A \oplus B)^+ = A^+ \oplus B^+$.
\end{itemize}
\end{proposition}
\begin{proof}
Let $A = U_1\Sigma_1V_1$ and $B = U_2\Sigma_2V_2$ be singular value decompositions for $A$ and $B$ over $\mathbb C$. By  Lemma \ref{lem:properties_kroneckerprod_directsum}, we then have that 
\begin{equation}
    A\otimes B = (U_1\otimes U_2)(\Sigma_1\otimes \Sigma_2)(V_1\otimes V_2)
\end{equation}
and similarly
\begin{equation}
    A\oplus B = (U_1\oplus U_2)(\Sigma_1 \oplus \Sigma_2)(V_1\oplus V_2)
\end{equation}
with $U_1\otimes U_2, V_1\otimes V_2, U_1\oplus U_2, V_1\oplus V_2$ all unitary and $\Sigma_1\otimes \Sigma_2, \Sigma_1 \oplus \Sigma_2$ both diagonal. Observe also that 
$(\Sigma_1\otimes \Sigma_2)^+ = \Sigma_1^+ \otimes \Sigma_2^+$ and $(\Sigma_1\oplus \Sigma_2)^+= \Sigma_1^+\oplus \Sigma_2^+$. 
Equation (\ref{eq:SVD_pseudoinverse}) then gives
\begin{equation*}
    (A\otimes B)^+ = (U_1\otimes U_2)^*(\Sigma_1\otimes \Sigma_2)^+(V_1\otimes V_2)^*
    = (U_1^*\Sigma_1^+V_1^*) \otimes (U_2^*\Sigma_2^+V_2^*) 
    = A^+ \otimes B^+
\end{equation*}
and similarly for $(A\oplus B)^+$. 
\end{proof}

By induction, we may generalize Proposition \ref{prop: kronecker_direct_sum} to finitely many matrices. 
\begin{corollary}
Let $n \in \mathbb{N}$, $\{A_i\}_{i=1}^n$ be a finite collection of matrices over $\mathbb{C}$. Then,
\begin{itemize}
\item[(i)] $(A_1 \otimes \ldots \otimes A_n)^+
 = A_1^+ \otimes \ldots \otimes A_n^+$,
 \item[(ii)] $(A_1 \oplus \ldots \oplus A_n)^+
 = A_1^+ \oplus \ldots \oplus A_n^+$. 
\end{itemize}
\end{corollary}

\section{Weightings and magnitude via pseudoinversion}\label{sec:magnitude_via_pseudoinversion}

We first remark that any generalized inverse of a matrix $M$ gives an explicit way of finding matricial weightings and coweightings, and thus a formula for the magnitude of $M$.

\begin{theorem}\label{thm:formula_magnitude_matrix}
Let $M$  be an element of $\text{Mat}_{m\times n}(\subfield)$ and let $M'\in \text{Mat}_{n\times m}(\subfield)$ be a generalized inverse of $M$.  The matrix $M$ has a weighting (resp. coweighting) if and only if $M'1_m$ is a matricial weighting (resp. $(M')^T1_n$ is a matricial coweighting) of $M$. In particular, if $M$ has magnitude, then 
\begin{equation}
\label{eq: Euler_char_pseudoinverse_matrix}
\chi(M) = \sum_{i=1}^n \sum_{j=1}^m M'_{i,j}
\end{equation}
\end{theorem}
\begin{proof}
To prove the first claim, suppose that 
$M$ has a matricial weighting. Take notation as in Definition \ref{def: weight_matrix} and let $w$ be the corresponding weighting vector. Then, $Mw = 1_m$. By associativity of matrix multiplication and the equality $MM'M=M$, we then have that
\begin{equation*}
M(M'1_m) = (MM')1_m = MM'Mw = Mw = 1_m.
\end{equation*}
So $M'1_m$ is matricial weighting on $M$. 

Similarly, assume that $M$ has a coweighting $z$. Then, $z^TM = 1_n^T$. In this case, 
\begin{equation*}
(1_n^TM')M = z^TMM'M = z^TM=1_n^T.
\end{equation*}
Therefore $(1^T_nM')^T = (M')^T1_n$ is a matricial coweighting of $M$. 

If $M$ has magnitude, then by definition, it has both a weighting and a coweighting; in particular $w=M'1_m$ is a weighting. We then see that 
\begin{equation}
\chi(M) = \sum_{i=1}^n w_i 
= \sum_{i=1}^n \sum_{j=1}^m M'_{i,j}.
\end{equation}
\end{proof}
\begin{remark}
 A first version of this theorem assumed that $M$ was square and $M'$ was the Moore-Penrose pseudoinverse, but the anonymous reviewer pointed out that these assumptions could be weakened. Remark that our proof also works for matrices over integral domains, so Theorem \ref{thm:formula_magnitude_matrix} might be applied even when the Moore-Penrose pseudoinverse does not exist; see Remark \ref{rmk:integral_domains}.
 \end{remark}

\begin{remark}
  Although Theorem \ref{thm:formula_magnitude_matrix} and the later results in this paper provide general tools for checking if a finite category has magnitude and subsequently calculating it, when using numerical computations for concrete examples, solving the system given by \eqref{eq: weighting} and \eqref{eq: coweighting} might be more efficient.
\end{remark}

Now we apply this to the categorical setting. Recall that we may identify $R(\cat{A})$ with the algebra of $(n \times n)$-matrices over $\mathbb{Q}$, where $n = |\Ob(\cat{A})|$. Let $f \in R(\cat{A})$ with corresponding matrix $F \in \text{Mat}_{n\times n} (\mathbb{Q})$. By Lemma \ref{lem: field_pseudo}, we then have that $F^+ \in \text{Mat}_{n\times n} (\mathbb{Q})$, and thus $F^+$ determines an element $f^+ \in R(\cat{A})$. We consider this process for $\zeta \in R(\cat{A})$.

\begin{definition}\label{def:pseudomoebius}
Let $\cat{A} \in \cat{FinCat}$ and let $Z$ be its zeta matrix (under some ordering of the objects). We then define the \emph{pseudo-M\"{o}bius function} $\tilde{\mu} \in R(\cat{A})$ as the element corresponding to the pseudoinverse $Z^+$, i.e. $\tilde{\mu}(x_i,x_j) = Z^+_{i,j}$ for all $x_i,x_j \in\Ob(\cat{A})$. 
\end{definition}

\begin{remark}
The definition of $\tilde{\mu}$ does not depend on the ordering used to define $Z$. To see this, let $\sigma \in S_n$ be a permutation and define another zeta matrix $Z'$ by $Z'_{i,j} = \zeta(x_{\sigma(i)}, x_{\sigma(j)}) = Z_{\sigma(i),\sigma(j)}$. 
Then, $Z' = P_{\sigma}^{-1}ZP_{\sigma}$ where $P_{\sigma} \in \text{Mat}_{n\times n}(\mathbb{Q})$ is the matrix obtained from the identity matrix by permuting the columns by $\sigma$. By \cite[Lem~1.6]{Penrose}, it follows that $(Z')^+ = P_{\sigma}^{-1}Z^+P_{\sigma}$. Then $(Z')^+_{i,j} = Z^+_{\sigma(i), \sigma(j)}$, so the pseudo-M\"{o}bius functions defined by $(Z')^+$ and $ Z^+$ agree. 
\end{remark}

\begin{remark}[Magnitude of metric spaces]
    As above, let $\subfield$ be a subfield of $\mathbb C$. Given a finite category $\cat A$ enriched over a monoidal category $(\cat V,\otimes,1)$ and a monoid morphism $\norm{\cdot}:(\cat V/\cong,\otimes,1) \to (\mathbb K,\cdot,1)$, defined over isomorphism classes of $\cat V$, reference \cite{leinster2013magnitude} defines the  zeta function $\zeta:\Ob(\cat A)\times \Ob(\cat A)\to \subfield$  via $\zeta(a,b)=\norm{\Hom(a,b)}$ . This can be represented by a matrix $Z\in \text{Mat}_{n\times n} (\subfield) $, which has a unique Moore-Penrose pseudoinverse. Therefore our definition of pseudo-M\"obius function, and all the results below, extend to this enriched setting. 
    
    In particular, when $\cat V$ is the category whose objects are $[0,\infty]$ and such that $x\to y$ iff $x\geq y$, a $\cat V$-enriched category is a generalized metric space (in the sense of Lawvere); the distance between two points is, by definition,  $d(a,b):=\Hom(a,b)\in [0,\infty]$. Every nonzero valuation is of the form $\norm{\Hom(a,b)}=\exp(-t d(a,b))$ for some $t\in \mathbb R$. The magnitude of any matrix $Z$ representing $\zeta$ is then an isometric invariant of metric spaces. For details, see \cite{leinster2013magnitude}.
\end{remark}

\begin{remark}
    In the case of posets, the zeta function is always invertible and $\zeta(x,y)=0$ implies $\mu(x,y)=0$ \cite[Thm. 4.1]{Leinster}, hence $\mu$ is also an element of the fine incidence algebra (see Remark \ref{rem:incidence_algs}). For general categories, one might have $\tilde \mu(x,y) \neq 0$ even if there are no arrows from $x$ to $y$. 
\end{remark}

There is an obvious corollary of Theorem \ref{thm:formula_magnitude_matrix}.

\begin{corollary}\label{cor:computation_Euler_char}
Let $\cat{A} \in \cat{FinCat}$. Then, $\cat{A}$ has a weighting (resp. coweighting) if and only if $k^x = \sum_{y\in\Ob(\cat{A})} \tilde{\mu}(x,y)$ is a weighting (resp. $k_x = \sum_{y\in\Ob(\cat{A})} \tilde{\mu}(y,x)$ is a coweighting) on $\cat{A}$. In particular, if $\cat{A}$ has magnitude, then 
\begin{equation}
\label{eq: Euler_char_pseudo}
\chi(\cat{A}) = \sum_{x,y\in\Ob(\cat{A})}\tilde{\mu}(x,y).
\end{equation}
\end{corollary}

So the Moore-Penrose pseudoinverse  gives us a way of checking if a finite category $\cat{A}$ has a weighting, coweighting, and thereby magnitude, along with a formula for calculating $\chi(\cat{A})$ in the case that $\cat{A}$ has magnitude. According to Theorem \ref{thm:formula_magnitude_matrix}, this is also true for any generalized inverse of $Z$.   However, the uniqueness of the  Moore-Penrose pseudoinverse is essential for Definition \ref{def:pseudomoebius}, and the properties of the pseudoinverse discussed in Section \ref{sec:pseudoinverse} translate into formulas for the pseudo-M\"obius function of products and coproducts of categories in terms of the pseudo-M\"obius functions of their factors.

Moreover, because the Moore-Penrose pseudoinverse of a matrix over $\mathbb C$ always exists, we can extend the definition of magnitude to \emph{any} rectangular matrix and \emph{any} finite category.

\begin{definition}
    The  \emph{generalized magnitude} of an $(m\times n)$-matrix $M$ over a subfield $\subfield$ of $\mathbb C$ is 
    \begin{equation}
        \chi(M) =\sum_{i=i}^n \sum_{j=1}^m M^+_{i,j}.
    \end{equation}
    The \emph{generalized magnitude} of a finite category $\cat A$ (possibly enriched) is the magnitude of any matrix representing its zeta function $\zeta$. 
\end{definition}

Because of Theorem 4.1, the generalized magnitude coincides with Leinster's magnitude (Definition \ref{def:magnitude}) whenever the latter makes sense i.e. when the matrix or category has a weighting and a coweighting. Hence we denote it with the same letter, $\chi$. 

Remark also that if a matrix $M$ has coefficients in a subfield $\subfield$ of $\mathbb C$, then $\chi(M)$ belongs to $\subfield$ too in virtue of Lemma \ref{lem: field_pseudo}.

We now consider the behavior of the pseudo-M\"{o}bius function and the generalized magnitude under taking the product and coproduct of categories. 
Let $\cat{A}, \cat{B} \in \cat{FinCat}$, and let $\zeta_{\cat{A}}, \zeta_{\cat{B}}$ denote their respective zeta functions. Similarly, let $Z_{\cat{A}}, Z_{\cat{B}}$ denote the corresponding zeta matrices (under some ordering of the objects). Let $\cat{A}\times\cat{B}$ and $\cat{A} + \cat{B}$ denote the product \cite[II.3]{MacLane1998} and coproduct (disjoint union) respectively. 
Observe that for all $x,y\in\Ob(\cat{A}), a,b \in\Ob(\cat{B})$,
\begin{equation}
\zeta_{\cat{A}\times \cat{B}}(\langle x,a\rangle, \langle y,b \rangle)
= \zeta_{\cat{A}}(x,y)\zeta_{\cat{B}} (a,b).
\end{equation}
Similarly, for any $x,y \in\Ob(\cat{A} + \cat{B})$,
\begin{equation}
\zeta_{\cat{A} + \cat{B}} (x,y) = 
\begin{cases}
\zeta_{\cat{A}}(x,y) & x,y \in\Ob(\cat{A})\\
\zeta_{\cat{B}} (x,y) & x,y \in\Ob(\cat{B})\\
0 & \text{otherwise}
\end{cases}.
\end{equation}
Let $Z_{\cat{A} \times \cat{B}}, Z_{\cat{A} + \cat{B}}$ denote the zeta matrices of the product and coproduct respectively. We then see that $Z_{\cat{A}\times \cat{B}} = Z_{\cat{A}} \otimes Z_{\cat{B}}$ and $Z_{\cat{A} + \cat{B}} = Z_{\cat{A}} \oplus Z_{\cat{B}}$. Then by Proposition \ref{prop: kronecker_direct_sum}, $Z_{\cat{A}\times \cat{B}}^+ = Z_{\cat{A}}^+ \otimes Z_{\cat{B}}^+$ and $Z_{\cat{A} + \cat{B}}^+ = Z_{\cat{A}}^+ \oplus Z_{\cat{B}}^+$. We obtain in this way the following generalization of \cite[Lem. 1.13]{Leinster}.
\begin{lemma}
Let $\cat{A}, \cat{B} $ be finite categories, $\tilde{\mu}_{\cat{A}}, \tilde{\mu}_{\cat{B}}$  their respective pseudo-M\"{o}bius functions, and   $\tilde{\mu}_{\cat{A}\times\cat{B}}$ (resp. $ \tilde{\mu}_{\cat{A}+\cat{B}}$) the pseudo-M\"{o}bius function of $\cat{A} \times \cat{B}$ (resp. $\cat{A} + \cat{B}$). 
\begin{itemize}
\item[(i)] For all $x_1,y_1\in\Ob(\cat{A})$ and $ x_2,y_2 \in\Ob(\cat{B})$,   $$\tilde{\mu}_{\cat{A}\times\cat{B}}(\langle x_1,y_1\rangle, \langle x_2,y_2 \rangle) = \tilde{\mu}_{\cat{A}}(x_1,y_1)\tilde{\mu}_{\cat{B}}(x_2,y_2).$$
\item[(ii)] For all $x,y \in\Ob(\cat{A} + \cat{B})$, $$\tilde{\mu}_{\cat{A} + \cat{B}}(x,y) = 
\begin{cases}
\tilde{\mu}_{\cat{A}} (x,y) & x,y \in\Ob(\cat{A})\\
\tilde{\mu}_{\cat{B}} (x,y) & x,y \in\Ob(\cat{B})\\
0 & \text{otherwise}
\end{cases}.$$
\end{itemize}
\end{lemma}
By induction, we may again generalize this to a finite family of finite categories:
\begin{corollary}\label{cor:properties-pseudomoebius}
Let  $\{\cat{A}_i\}_{i=1}^n$ be collection of finite categories. 
\begin{itemize}
\item[(i)] For all $(x_i)_{i=1}^n, (y_i)_{i=1}^n \in \Ob(\prod_{i=1}^n \cat{A}_i)$, $$\tilde{\mu}_{\prod_{i=1}^n \cat{A}_i}((x_i)_{i=1}^n, (y_i)_{i=1}^n) = \prod_{i=1}^n \tilde{\mu}_{\cat{A}_i}(x_i,y_i).$$ 
\item[(ii)] For all $x\in\Ob(\cat{A}_i), y\in\Ob(\cat{A}_j)$, $$\tilde{\mu}_{\sum_{i=1}^n \cat{A}_i}(x,y) = 
\begin{cases}
\tilde{\mu}_{\cat{A}_i} (x,y) & x,y \in \Ob(\cat{A}_i)\\
0 & \text{otherwise}
\end{cases}.$$ 
\end{itemize}
\end{corollary}
In turn, Corollary \ref{cor:properties-pseudomoebius} gives us a similar property of the generalized magnitude, that extends  \cite[Prop~2.6]{Leinster} to arbitrary finite categories. 
\begin{lemma}
Let $\{\cat{A}_i\}_{i=1}^n$ be a finite collection of finite categories. Then, 
\begin{itemize}
\item[(i)] $\chi(\prod_{i=1}^n \cat{A}_i) = \prod_{i=1}^n \chi(\cat{A}_i)$
\item[(ii)] $\chi(\sum_{i=1}^n \cat{A}_i) = \sum_{i=1}^n \chi(\cat{A}_i)$
\end{itemize}
\end{lemma}

\section*{Acknowledgments}

S.C. acknowledges the support of Marcella Bonsall through her SURF fellowship. J.P.V.  thanks Matilde Marcolli for several conversations about this subject. Both authors also acknowledge the multiple suggestions made by the anonymous referee, concerning mainly the validity of the results for rectangular matrices, the possibility of proving Proposition \ref{prop: kronecker_direct_sum} using the SVD, and  the validity of Theorem \ref{thm:formula_magnitude_matrix} when $M'$ is just a generalized inverse. 

\bibliographystyle{unsrt}
\bibliography{SURF_2022_BIB}

\end{document}